\documentclass[12pt,a4paper]{article}

\usepackage[margin=1in]{geometry}
\usepackage[numbers,sort&compress]{natbib}
\usepackage{amssymb,amsthm,mathtools}
\usepackage{enumitem}
\usepackage[hidelinks]{hyperref}

\newtheorem{theorem}{Theorem}[section]
\newtheorem{proposition}[theorem]{Proposition}
\newtheorem{lemma}[theorem]{Lemma}

\newtheorem{remark}[theorem]{Remark}

\begin{document}

\pagestyle{myheadings}
\markright{First Modular Zhu Algebra in Characteristic 7}

\title{The First Modular Zhu Algebra of
\(L(\frac12,0)_F\) in Characteristic Seven}

\author{Haihua Huang\\
\small School of General Education, Guangzhou Civil Aviation College,\\
\small Guangzhou 510403, Guangdong, China\\
\small Corresponding author: Haihua Huang, \texttt{hoiwaa@163.com}\\
\small ORCID: 0009-0008-7337-3947}

\date{}

\maketitle

\begin{abstract}
Let \(F\) be an algebraically closed field of characteristic \(7\), and let
\(V=L(\frac12,0)_F\). For the degree-four singular vector
\(q=(L(-2)^2-2L(-4))\mathbf1\) in the universal vacuum vertex operator
algebra \(\bar V\), set \(V_q=\bar V/I(q)\). We prove that the natural
quotient map induces an isomorphism
\(A_1(V_q)\cong A_1(V)\cong F[T]/(T(T-4)(T-5))\cong F^3\).
Explicit \(O_1(V_q)\)-relations and four families of triangular PBW
straightening relations give the upper bound, while three one-dimensional
homogeneous subspaces give the matching lower bound.
\end{abstract}


\section{Introduction}

Zhu's algebra \(A(V)\) encodes the top homogeneous subspace of an
admissible \(V\)-module as a module over an associative algebra
\citep{Zhu96}. This construction underlies the analysis of rationality for
Virasoro minimal models~\citep{Wang93}. Over the complex numbers, the
higher level algebras \(A_n(V)\), which act on the first \(n+1\)
homogeneous subspaces of an admissible module, were introduced
in~\citep{DLM98}; their finite-dimensional semisimplicity is also tied
to rationality. In particular, \(A_1(V)\) acts on the first two
homogeneous subspaces, whereas \(A_0(V)\) acts only on the top level.
Further module-theoretic developments of higher level Zhu algebras appear
in~\citep{BVY19}. The level-one Zhu algebra of the Virasoro vertex
operator algebra associated with a generalized Verma module over
\(\mathbb C\) was computed explicitly in~\citep{BVY20}. The present
problem instead concerns the modular algebra of a simple quotient in
characteristic \(7\) and its degree-four singular-vector quotient. Zhu's
theory over an arbitrary algebraically closed field was established
in~\citep{DR14}, and the corresponding modular
\(A_n(V)\)-theory was constructed in~\citep{Ren17}; related results for vertex
operator superalgebras appear in~\citep{DW22}. In positive characteristic,
binomial coefficients in the \(O_n(V)\)-relations may vanish, distinct
integral degrees may define the same \(L(0)\)-eigenvalue, and reduction
may introduce additional singular vectors. Ren's definition retains the
full family of circle-product relations
\(u\circ^s_{n,t}v\) with \(0\le s\le t\), together with the translation
relations~\citep{Ren17}. Modular calculations therefore use this full family of
\(O_n(V)\)-relations. The action on homogeneous subspaces gives a
representation-theoretic quotient of \(A_n(V)\); an independent spanning
argument is needed to determine the dimension of the algebra. At level
\(n=1\), the relations used to eliminate high Virasoro modes come from
Ren's full family of \(O_1(V)\)-relations rather than from the ordinary
Zhu relation alone.

For the Virasoro algebra, the role of singular vectors in highest weight
modules is classical~\citep{Kac90,IK11,Ast97}. Let \(F\) be an
algebraically closed field of characteristic \(7\), and set
\(V=L(\frac12,0)_F\). Here
\(\frac12=\frac1{16}=4\in F\), so two highest weights from the complex
theory coincide. It was shown in~\citep{DR16} that the irreducible
admissible \(V\)-modules have highest weights \(0\) and \(4\), and that
the universal vacuum vertex operator algebra
\(\bar V=\bar V(\frac12,0)_F\) contains the degree-four singular vector
\(q=(L(-2)^2-2L(-4))\mathbf1\). This vector vanishes in the simple
quotient \(V\). For characteristics different from \(2\) and \(7\), the
same simple Virasoro vertex operator algebra is rational and has three
inequivalent irreducible modules~\citep{DR16}; the coincidence of weights
and the additional degree-four relation make characteristic \(7\) the
exceptional case. Further work~\citep{DLR24} showed that, in
characteristic \(7\), the Virasoro subalgebra generated by the conformal
vector is properly contained in the full even fermionic vertex operator
algebra \(V(H_F,\mathbb Z+\frac12)^{\bar0}\). They calculated the ordinary
Zhu algebra of the full even algebra and classified its irreducible
modules. Accordingly, the present calculation concerns the higher Zhu
algebra of the simple Virasoro subalgebra \(V\).

The ordinary Zhu algebra provides the baseline
\(A_0(V)\cong F[T]/(T(T-4))\cong F\times F\). The standard Virasoro Zhu
calculation~\citep{Wang93}, used in the arbitrary-field framework
of~\citep{DR14}, shows that this algebra is generated by the class of
\(\omega\). The degree-four relation and the modules in~\citep{DR16} then
give the polynomial \(T(T-4)\) and separate its two roots. For \(A_1(V)\),
denote the two irreducible modules
by \(\mathcal L_0\) and \(\mathcal L_4\). The homogeneous subspaces
\(\mathcal L_0(0)\), \(\mathcal L_4(0)\), and \(\mathcal L_4(1)\) are
one-dimensional, and the zero mode of \(\omega\) acts on them with
eigenvalues \(0,4,5\), respectively. They yield a surjective algebra
homomorphism \(A_1(V)\to F^3\). Let \(I(q)\) be the graded ideal of
\(\bar V\) generated by \(q\), and set \(V_q=\bar V/I(q)\). We prove that
the natural quotient map \(\pi_V:V_q\to V\) induces an algebra isomorphism
\(A_1(\pi_V):A_1(V_q)\to A_1(V)\), and that
\(A_1(V)\cong F[T]/(T(T-4)(T-5))\cong F^3\). The upper bound is
established in \(A_1(V_q)\), which isolates the degree-four relation from
the remaining relations of the simple quotient. The relation \(q=0\)
eliminates pairs of \(L(-2)\)-factors, and the \(O_1(V_q)\)-relations
replace modes \(L(-m)\) with \(m\ge5\) by lower-index modes. Three
triangular families coming from the translation relation handle all
resulting monomials except \(L(-4)^kL(-2)\mathbf1\) with \(k\ge2\). A
fourth family, obtained from a circle-product relation, handles these
remaining monomials. Together they reduce all PBW classes to the span of
\(\mathbf1\), \(L(-2)\mathbf1\), and
\(L(-4)L(-2)\mathbf1\). This proves \(\dim_F A_1(V_q)\le3\); the induced
map \(A_1(V_q)\to A_1(V)\) and the three one-dimensional homogeneous
subspaces then give equality throughout.

Section~\ref{sec:prelim} recalls the Virasoro and modular \(A_n(V)\)
notation used below. Section~\ref{sec:main-proof} establishes the module
and PBW calculations and concludes with the main theorem.

\section{Preliminaries}
\label{sec:prelim}

\subsection{Virasoro vertex operator algebras and PBW notation}

Throughout the paper, \(F\) denotes an algebraically closed field of
characteristic \(7\). Since \(2\) and \(12\) are invertible in \(F\), the
Virasoro algebra of central charge \(c=\frac12=4\in F\) is well defined
over \(F\)~\citep{DR16}. Let \(\mathrm{Vir}_F\) denote the Virasoro algebra
over \(F\), with generators \(L(n)\) for \(n\in\mathbb Z\) and central
element \(C\), subject to
\[
[L(m),L(n)]
=(m-n)L(m+n)+\delta_{m+n,0}\frac{m^3-m}{12}C,\qquad [C,L(n)]=0.
\]

We use the conventions of~\citep{DR14} for vertex operator
algebras and their modules over an arbitrary field. For homogeneous
\(a\in V_r\), write \(\operatorname{wt}(a)=r\), and set
\(Y(a,z)=\sum_{n\in\mathbb Z}a_nz^{-n-1}\) and
\(Y(\omega,z)=\sum_{n\in\mathbb Z}L(n)z^{-n-2}\). Following the notation
of~\citep{DR16}, for \(c,h\in F\), let \(V(c,h)_F\) denote the
Virasoro Verma module of central charge \(c\) and highest weight \(h\),
with highest weight vector \(v_{c,h}\), and let \(L(c,h)_F\) denote its
unique irreducible graded quotient. For \(h=0\), set
\(
\bar V(c,0)_F
:=V(c,0)_F/U(\mathrm{Vir}_F)L(-1)v_{c,0},
\)
and denote the image of \(v_{c,0}\) by \(\mathbf1\). The PBW basis of
\(\bar V(c,0)_F\) consists of
\(L(-n_1)\cdots L(-n_r)\mathbf1\), where
\(n_1\ge\cdots\ge n_r\ge2\), with \(r=0\) corresponding to
\(\mathbf1\).

In particular, \(\bar V=\bar V(\frac12,0)_F\) is the universal vacuum
Virasoro vertex operator algebra of central charge \(\frac12\).
Let \(J\) be its unique maximal proper graded ideal. The object studied
in this paper is the simple quotient
\(V=L(\frac12,0)_F=\bar V/J\).

An admissible \(V\)-module is a graded weak \(V\)-module
\(M=\bigoplus_{r\ge0}M(r)\). For every homogeneous \(a\in V\), every
\(j\in\mathbb Z\), and every \(r\ge0\), the grading satisfies
\(a_jM(r)\subseteq M(r+\operatorname{wt}(a)-j-1)\),
where \(M(k)=0\) for \(k<0\)~\citep{DR14}. We call \(M(r)\) the degree-\(r\)
homogeneous subspace of \(M\). Here both \(r\) and
\(\operatorname{wt}(a)\) always refer to the integral grading.

For \(m_i\ge2\), let
\(Q=L(-m_1)\cdots L(-m_r)\mathbf1\) be a vacuum word; when \(r=0\),
this means \(Q=\mathbf1\). If \(m_1\ge\cdots\ge m_r\), then \(Q\) is
called a vacuum PBW monomial. Rewriting an arbitrary vacuum word as a
linear combination of vacuum PBW monomials by means of the Virasoro
commutation relations will be called putting the word into PBW normal
form. Let \(|Q|:=m_1+\cdots+m_r\) be the integer total degree and
\(\operatorname{len}(Q):=r\) the PBW length. Every PBW monomial
occurring in the PBW normal form of \(Q\) has length at most
\(\operatorname{len}(Q)\).

\subsection{Modular higher level Zhu algebras}

The modular higher level Zhu algebra theory over an arbitrary
algebraically closed field was developed in~\citep{Ren17}. Fix
\(n\in\mathbb Z_{\ge0}\). For homogeneous \(u\in V\), arbitrary
\(v\in V\), and nonnegative integers \(s,t\) satisfying \(0\le s\le t\),
define
\[
u\circ^s_{n,t}v
:=
\operatorname{Res}_z\,Y(u,z)v
\frac{(1+z)^{\operatorname{wt}(u)+n+s}}{z^{2n+2+t}}.
\]

Let \(O_n(V)\) be the linear span of all vectors \(u\circ^s_{n,t}v\) as
above and all vectors \((L(-1)+L(0))\xi\), where \(\xi\) is homogeneous,
and set \(A_n(V):=V/O_n(V)\). Define
\[
u*_n v
:=
\sum_{m=0}^{n}(-1)^m\binom{m+n}{n}
\operatorname{Res}_z\,Y(u,z)v\,
\frac{(1+z)^{\operatorname{wt}(u)+n}}{z^{n+m+1}}.
\]
Both operations are extended linearly to nonhomogeneous vectors by
homogeneous decomposition. For \(a\in V\), write
\([a]_n:=a+O_n(V)\). If \(a\) is homogeneous, define its zero mode by
\(o(a):=a_{\operatorname{wt}(a)-1}\), and extend \(o\) linearly to \(V\)
by homogeneous decomposition.

\begin{remark}
The displayed definition in~\citep{Ren17} states
\(s,t\in\mathbb Z\) with \(s\le t\), whereas the discussion of
\(A_0(V)=A(V)\) and the proof that \(O_n(V)\) acts trivially on
\(\Omega_n(M)\) use \(t\ge s\ge0\). The proof that \(O_n(V)\) is stable
under the product \(*_n\) likewise begins with generators satisfying
\(t\ge s\ge0\). Thus the proofs of the associative algebra structure and
the action on \(\Omega_n(M)\) establish these results for the latter
range, which we adopt throughout. This range is also required for
compatibility with \(A_0(V)=A(V)\): if \(s=t=-1\) were allowed, then for
every \(v\in V\), one would have
\(\mathbf1\circ^{-1}_{0,-1}v
=\operatorname{Res}_z v\,(1+z)^{-1}/z=v\),
which would give \(O_0(V)=V\).
\end{remark}

The following result supplies the associative algebra structure and the
actions on low homogeneous subspaces that will be used throughout the
paper.

\begin{theorem}
\label{thm:modular-An-well-defined}
\citep{Ren17}
The product \(*_n\) induces an associative algebra structure on
\(A_n(V)\), with identity \([\mathbf1]_n\) and central element
\([\omega]_n\). If \(M\) is a weak \(V\)-module, then
\[
\Omega_n(M)=
\{\eta\in M\mid
u_i\eta=0\text{ for all homogeneous }u\in V
\text{ and all }i\ge\operatorname{wt}(u)+n\}
\]
is an \(A_n(V)\)-module under the action \([a]_n\mapsto o(a)\).
In particular, if \(M=\bigoplus_{r\ge0}M(r)\) is an admissible module,
then, for \(0\le r\le n\), \(M(r)\) is an \(A_n(V)\)-submodule of
\(\Omega_n(M)\).
\end{theorem}

For the explicit calculations below, we record the two families of
\(O_n(V)\)-relations that will be used repeatedly.

\begin{lemma}
\label{lem:higher-Zhu-On-relations}
\citep{Ren17}
Let \(n\ge0\). In \(A_n(V)=V/O_n(V)\), one has
\((L(-1)+L(0))u\in O_n(V)\) for every homogeneous \(u\in V\).
Moreover, for homogeneous \(a,b\in V\) and nonnegative integers \(s,t\)
with \(0\le s\le t\),
\[
a\circ^s_{n,t}b
:=
\operatorname{Res}_zY(a,z)b\,
\frac{(1+z)^{\operatorname{wt}(a)+n+s}}{z^{2n+2+t}}
\in O_n(V).
\]
In particular, taking \(a=\omega\), \(b=u\), and \(s=0\), one obtains,
for every \(t\ge0\), the congruence
\[
\sum_{j=0}^{n+2}\binom{n+2}{j}
L(-(2n+t+3-j))u\equiv0
\]
in \(A_n(V)\).
\end{lemma}

The upper bound will first be proved for a singular-vector quotient and
then transferred to its simple quotient. We shall use the functoriality
of \(A_n\) under quotient maps. Let \(f:U\to U'\) be a surjective
homomorphism of vertex operator algebras preserving the vacuum and
conformal vectors. For the generators of \(O_n(U)\), one has
\(f(a\circ^s_{n,t}b)=f(a)\circ^s_{n,t}f(b)\) and
\(f((L(-1)+L(0))\xi)=(L(-1)+L(0))f(\xi)\), and hence
\(f(O_n(U))\subseteq O_n(U')\). Moreover,
\(f(a*_nb)=f(a)*_nf(b)\). Therefore, for every \(n\ge0\), the map \(f\)
induces a surjective homomorphism of unital algebras
\(A_n(f):A_n(U)\to A_n(U')\).

\section{Proof of the main result}
\label{sec:main-proof}

\subsection{Low homogeneous subspaces and PBW straightening}

The vector \(q=(L(-2)^2-2L(-4))\mathbf1\) is a singular vector of degree
\(4\) in \(\bar V\)~\citep{DR16}. Let
\(I(q)=U(\mathrm{Vir}_F)q\). The vertex operator algebra \(\bar V\) is
strongly generated by \(\omega\), and \(I(q)\) is stable under all modes
of \(\omega\), equivalently under all Virasoro modes. The mode iterate
formula~\citep{DR14} therefore implies stability under the modes of every element of
\(\bar V\), so \(I(q)\) is a graded vertex algebra ideal. Since \(q\) is
singular, the PBW theorem~\citep{DR16} gives
\(I(q)=U(\mathrm{Vir}_{F,-})q\), where
\(\mathrm{Vir}_{F,-}=\bigoplus_{m\ge1}F L(-m)\). This ideal is
concentrated in degrees at least \(4\), and is therefore proper. The
maximality of \(J\) gives \(I(q)\subseteq J\). Set
\(V_q:=\bar V/I(q)\). We therefore have
natural quotient maps \(\pi_q:\bar V\to V_q\) and
\(\pi_V:V_q\to V\).
For the calculations in \(A_1(V_q)=V_q/O_1(V_q)\), if \(a,b\in V_q\),
write \(a\equiv_{A_1}b\) when \(a-b\in O_1(V_q)\), or equivalently when
\([a]_1=[b]_1\) in \(A_1(V_q)\). If \(a,b\in\bar V\), the notation
\(a\equiv_{A_1}b\) means
\(\pi_q(a)-\pi_q(b)\in O_1(V_q)\).
The PBW upper bound will first be established for \(V_q\) and then
transferred to the simple quotient \(V\) through \(\pi_V\). The
representation-theoretic lower bound will be obtained directly from
\(V\)-modules.

We next recall the fermionic realization that provides the required
\(V\)-modules. Let \(H_F=Fa\) be a one-dimensional \(F\)-space with
symmetric bilinear form determined by \((a,a)=1\). Let
\(\mathcal F=V(H_F,\mathbb Z+\frac12)\) be the one-dimensional free
fermion vertex operator superalgebra, with parity decomposition
\(\mathcal F=\mathcal F^{\bar0}\oplus\mathcal F^{\bar1}\), where
\(\mathcal F^{\bar0}\) and \(\mathcal F^{\bar1}\) are its even and odd
subspaces, respectively. Its conformal vector is
\(\omega=\frac12a(-\tfrac32)a(-\tfrac12)\mathbf1\), and its central
charge is \(\frac12=4\in F\)~\citep{DR16,DLR24}.

The following realization identifies the simple Virasoro subalgebra and
provides the nonvacuum module needed for the lower-bound argument.

\begin{theorem}
\label{thm:fermionic-realization}
\citep{DR16}
Let \(V^{\mathrm{Vir}}\subset\mathcal F^{\bar0}\) be the vertex operator
subalgebra generated by \(\omega\), and let
\(\mathcal L_4\subset\mathcal F^{\bar1}\) be the
\(V^{\mathrm{Vir}}\)-submodule generated by
\(a(-\frac12)\mathbf1\). Then
\(V^{\mathrm{Vir}}\cong V=L(\frac12,0)_F\), and \(\mathcal L_4\) is the
irreducible admissible \(V\)-module of highest weight \(4\).
\end{theorem}

Write \(\mathcal F_r\) for the subspace of conformal degree \(r\), where
\(r\in\frac12\mathbb Z_{\ge0}\). Thus
\(\mathcal F=\bigoplus_r\mathcal F_r\).
As a \(V^{\mathrm{Vir}}\)-module, we grade \(\mathcal L_4\) admissibly
from its lowest conformal degree by
\(\mathcal L_4(m)=\mathcal L_4\cap\mathcal F_{m+\frac12}\) for
\(m\in\mathbb Z_{\ge0}\).
Thus \(a(-\frac12)\mathbf1\in\mathcal L_4(0)\), whereas
\(a(-\frac32)\mathbf1\in\mathcal L_4(1)\).

The analysis of generators for the full even fermionic algebra
in~\citep{DLR24} gives
\(V^{\mathrm{Vir}}\subsetneq\mathcal F^{\bar0}\) in characteristic \(7\).
Together with the preceding construction, the four objects used here are
related by the natural quotient maps and inclusion
\[
\bar V\to V_q\to
V\cong V^{\mathrm{Vir}}\subsetneq\mathcal F^{\bar0}.
\]

To obtain the three one-dimensional representations of \(A_1(V)\), we
now determine the first two homogeneous subspaces of the relevant
modules.

\begin{lemma}
\label{lem:low-layers}
Let \(\mathcal L_0=V\). Then
\[
\dim\mathcal L_0(0)=1,\qquad
\dim\mathcal L_0(1)=0,\qquad
\dim\mathcal L_4(0)=\dim\mathcal L_4(1)=1.
\]
\end{lemma}

\begin{proof}
The degree-zero subspace of the vacuum module is spanned by \(\mathbf1\),
while its degree-one subspace vanishes by the vacuum condition
\(L(-1)\mathbf1=0\).

For the shifted admissible grading above, set
\(v_4=a(-\frac12)\mathbf1\). The fermionic PBW basis~\citep{DR16} gives
\(\mathcal L_4(0)=Fv_4\) and
\(\mathcal L_4(1)=F a(-\tfrac32)\mathbf1=F L(-1)v_4\).
Indeed,
\([L(-1),a(-\tfrac12)]=a(-\tfrac32)\), and hence
\(L(-1)a(-\tfrac12)\mathbf1=a(-\tfrac32)\mathbf1\).
Therefore
\(\dim\mathcal L_4(0)=\dim\mathcal L_4(1)=1\).
\end{proof}

By Theorem~\ref{thm:modular-An-well-defined}, for \(h\in\{0,4\}\) and
\(0\le m\le n\), the low homogeneous subspace \(\mathcal L_h(m)\) is an
\(A_n(V)\)-module under the zero-mode action. Since
\(o(\omega)=L(0)\), the element \(x_n=[\omega]_n\) acts on
\(\mathcal L_h(m)\) with eigenvalue \(h+m\). The lower bound for
\(A_1(V)\) will therefore use
\(\mathcal L_0(0)\), \(\mathcal L_4(0)\), and \(\mathcal L_4(1)\), with
eigenvalues \(0,4,5\), respectively.

We now turn to the upper bound. The first simplification uses the
singular-vector relation to eliminate repeated \(L(-2)\)-factors.

\begin{lemma}
\label{lem:An-remove-two-L2}
Every vacuum word \(Q\) containing at least two \(L(-2)\)-factors is,
in \(V_q\), an \(F\)-linear combination of vacuum PBW monomials whose
lengths are strictly smaller than \(\operatorname{len}(Q)\).
Consequently, for every \(n\ge0\), the same conclusion holds in
\(A_n(V_q)\) and, via the natural quotient map, in \(A_n(V)\). Repeated
application therefore leaves at most one \(L(-2)\)-factor in each PBW
monomial.
\end{lemma}

\begin{proof}
Let \(Q=L(-k_1)\cdots L(-k_r)\mathbf1\) be such a vacuum word. For
\(m>2\), the Virasoro commutation relation gives
\(L(-2)L(-m)=L(-m)L(-2)+(m-2)L(-(m+2))\). The first term moves
\(L(-2)\) one place to the right without changing the length, whereas
the second term has one fewer negative-mode factor. Applying this
identity repeatedly to every \(L(-2)\)-factor gives
\(Q=R L(-2)^s\mathbf1+\sum_j Q_j\), where \(s\ge2\),
\(\operatorname{len}(R)+s=r\), and every vacuum word \(Q_j\) has length
strictly smaller than \(r\).

Set \(R'=R L(-2)^{s-2}\). Since \(q=0\) in \(V_q\), one has
\(R'L(-2)^2\mathbf1=2R'L(-4)\mathbf1\) in \(V_q\). Before PBW
normalization, the right-hand side has length
\(\operatorname{len}(R')+1=r-1\). Moreover, PBW normalization does not
increase length: an adjacent interchange produces a reordered word of
the same length together with a commutator term of smaller length, and
termination follows by induction on length and then on the number of
inversions. Hence the PBW normal forms of \(2R'L(-4)\mathbf1\) and all
\(Q_j\) contain only monomials of length strictly smaller than \(r\).
This proves the assertion in \(V_q\). Passing to \(A_n(V_q)\) and then
through the natural quotient map to \(A_n(V)\) proves the remaining
statements.
\end{proof}

After repeated \(L(-2)\)-factors have been eliminated, modes \(L(-m)\)
with \(m\ge5\) must be replaced by low-index modes. The following
relation provides this replacement.

\begin{lemma}
\label{lem:An-omega-high-index}
Fix \(n\ge0\), and let \(U\) be a vertex operator algebra. In
\(A_n(U)\), for every \(u\in U\) and every
\(m\ge2n+3\), the vector \(L(-m)u\) can be expressed recursively as an
\(F\)-linear combination of \(L(-(n+1))u,L(-(n+2))u,\ldots,
L(-(2n+2))u\).
\end{lemma}

\begin{proof}
By Lemma~\ref{lem:higher-Zhu-On-relations},
\(\omega\circ^0_{n,t}u\in O_n(U)\) for every \(t\ge0\), and its
class in \(A_n(U)\) is therefore zero. Here
\(
\omega\circ^0_{n,t}u
=\operatorname{Res}_zY(\omega,z)u
\frac{(1+z)^{n+2}}{z^{2n+2+t}}.
\)
Taking the residue gives
\(
\sum_{j=0}^{n+2}\binom{n+2}{j}
L(-(2n+t+3-j))u\equiv0.
\)
Putting \(m=2n+t+3\), this congruence becomes
\(
L(-m)u\equiv
-\sum_{j=1}^{n+2}\binom{n+2}{j}L(-(m-j))u.
\)
Every index on the right-hand side is smaller than \(m\). Induction on
\(m\) therefore expresses every term with \(m\ge2n+3\) successively in
terms of those with indices \(n+1,\ldots,2n+2\).
\end{proof}

In particular, when \(n=1\), every \(L(-m)u\) with \(m\ge5\) is an
\(F\)-linear combination, in \(A_1(U)\), of the three
negative-mode terms with indices \(2,3,4\). More explicitly, for every
\(u\in U\) and \(m\ge5\),
\[
L(-m)u\equiv-3L(-(m-1))u-3L(-(m-2))u-L(-(m-3))u.
\]
The boundary case \(m=5\), used repeatedly below, is
\[
L(-5)u\equiv -3L(-4)u-3L(-3)u-L(-2)u.
\]
We apply this \(n=1\) case with \(U=V_q\). Together with the
pair elimination in Lemma~\ref{lem:An-remove-two-L2}, it expresses the
class of every PBW monomial as a linear combination of PBW monomials
involving only \(L(-2),L(-3),L(-4)\) and containing at most one
\(L(-2)\)-factor.

\subsection{\texorpdfstring{The structure of \(A_1(V)\)}{The structure of A1(V)}}
\label{sec:A1}

In this subsection, \(F\) remains an algebraically closed field of
characteristic \(7\), and \(0,4,5\in\mathbb F_7\subset F\) are viewed as
elements of \(F\). We continue to write \(V=L(\frac12,0)_F\). As
admissible \(V\)-modules, write \(\mathcal L_0=V\) for the vacuum module.
The module \(\mathcal L_4\) constructed above is isomorphic to
\(L(\frac12,\frac12)_F\). Since \(\frac12=4\) in \(F\), the subscript
\(4\) records its highest weight. Write
\(x_q=[\omega]_1\in A_1(V_q)\) and
\(x=[\omega]_1\in A_1(V)\).

For the triangular induction below, we parametrize these PBW monomials
as follows. Let
\(
\Lambda_1=
\{(\lambda_4,\lambda_3,\lambda_2)\in\mathbb Z_{\ge0}^3:
\lambda_2\in\{0,1\}\}.
\)
For \(\lambda=(\lambda_4,\lambda_3,\lambda_2)\in\Lambda_1\), set
\(
P_\lambda=
L(-4)^{\lambda_4}L(-3)^{\lambda_3}
L(-2)^{\lambda_2}\mathbf1\in\bar V.
\)
The three coordinates record the multiplicities of
\(L(-4),L(-3),L(-2)\), respectively. Thus \(\Lambda_1\) indexes this
family of PBW monomials. Define the lexicographic ordering parameter on
these monomials by
\[
\nu(P_{(\lambda_4,\lambda_3,\lambda_2)})
=(\lambda_3+\lambda_4,\lambda_2,-\lambda_3),
\]
and call a monomial lower if its \(\nu\)-parameter is smaller. Accordingly,
define
\[
\mathcal P_{<P_\lambda}
=\operatorname{span}_F
\{P_\mu:\mu\in\Lambda_1,\ \nu(P_\mu)<\nu(P_\lambda)\}
\]
and set \(\mathcal B_1=
\operatorname{span}_F\{P_{(0,0,0)},P_{(0,0,1)},P_{(1,0,1)}\}\).
We refer to these three monomials as the retained monomials. In
congruences modulo \(O_1(V_q)\), we suppress \(\pi_q\) and write
\(P_\lambda\) for \(\pi_q(P_\lambda)\).

The calculations below require both control of PBW length under modes
that annihilate the vacuum and comparison with the order determined by
\(\nu\). The next lemma records these two estimates.

\begin{lemma}
\label{lem:A1-length-order}
Let \(Q=L(-m_1)\cdots L(-m_r)\mathbf1\) be a vacuum PBW monomial, where
\(m_1\ge\cdots\ge m_r\ge2\), and let \(p\in\mathbb Z\) with \(p\ge-1\).
After moving \(L(p)\) to the right and putting the result into vacuum PBW
normal form, every nonzero PBW monomial occurring in \(L(p)Q\) has length
at most \(r\); every term arising from a central commutator has length at
most \(r-1\). Moreover, let
\(\lambda=(\lambda_4,\lambda_3,1)\in\Lambda_1\) and
\(\mu\in\Lambda_1\). If
\(\operatorname{len}(P_\mu)\le\lambda_3+\lambda_4\), then
\(P_\mu\in\mathcal P_{<P_\lambda}\).
\end{lemma}

\begin{proof}
We prove the first assertion by induction on \(r\). For \(r=0\), it
follows from \(L(p)\mathbf1=0\) for \(p\ge-1\). For \(r>0\), write
\(Q=L(-m_1)Q'\). The Virasoro relation gives
\[
L(p)Q
=(p+m_1)L(p-m_1)Q'
+\delta_{p,m_1}\frac{p^3-p}{12}C Q'+L(-m_1)L(p)Q'.
\]
If \(p-m_1\le-2\), the first term is a word of length \(r\); if
\(p-m_1\ge-1\), the induction hypothesis applied to \(Q'\) gives length
at most \(r-1\). The central term has length \(r-1\). The induction
hypothesis also shows that the PBW monomials in \(L(p)Q'\) have length at
most \(r-1\); adjoining \(L(-m_1)\) and restoring PBW order therefore
gives length at most \(r\). Any central term arising in the induction
step for \(L(p)Q'\) has length at most \(r-2\) before \(L(-m_1)\) is
adjoined, and hence at most \(r-1\) afterward. PBW normalization does not
increase length, so the first assertion follows.

Set \(d=\lambda_3+\lambda_4\). By definition,
\(\operatorname{len}(P_\mu)=\mu_4+\mu_3+\mu_2\), while
\(\nu(P_\lambda)=(d,1,-\lambda_3)\). If \(\mu_2=1\), then
\(\mu_3+\mu_4\le d-1\), so the first coordinate of \(\nu(P_\mu)\) is
smaller than \(d\). If \(\mu_2=0\), then
\(\mu_3+\mu_4\le d\). A strict inequality is again decided by the first
coordinate, whereas equality is decided by the second coordinate
\(0<1\). Thus in both cases
\(\nu(P_\mu)<\nu(P_\lambda)\), proving the assertion.
\end{proof}

By the vacuum PBW theorem~\citep{DR16}, the vectors
\(L(-i_1)\cdots L(-i_r)\mathbf1\), where
\(i_1\ge\cdots\ge i_r\ge2\),
form a PBW basis of \(\bar V\), where \(r=0\) corresponds to
\(\mathbf1\). Let \(\mathcal P_1=
\operatorname{span}_F\{P_\lambda:\lambda\in\Lambda_1\}\subset\bar V\).
We now define the normalization map
\(\mathcal N_1\colon\bar V\to\mathcal P_1\) recursively, first by total
degree and then, for a fixed total degree, by PBW length. Let
\(Q=L(-i_1)\cdots L(-i_r)\mathbf1\) be a PBW basis vector. Its value is
determined by the following rules; the branches are interpreted in the
order in which they are listed. Every word occurring on a right-hand
side is first put into its unique PBW normal form in \(\bar V\), and the
rules are then applied again to the resulting PBW basis vectors:
\[
\mathcal N_1(Q)=
\begin{cases}
\Phi_m(Q'),
& \text{if }Q=L(-m)Q',\ m\ge5,\\
2\mathcal N_1(RL(-4)\mathbf1),
& \text{if }Q=R L(-2)^2\mathbf1
  \quad\text{and the first case does not apply},\\
P_\lambda,
& \text{if }Q=P_\lambda,\ \lambda\in\Lambda_1,
\end{cases}
\]
where
\[
\Phi_m(Q')=-3\mathcal N_1(L(-(m-1))Q')
-3\mathcal N_1(L(-(m-2))Q')-\mathcal N_1(L(-(m-3))Q').
\]

For a general element \(\xi\in\bar V\), take its unique PBW expansion
\(\xi=\sum_\alpha c_\alpha Q_\alpha\) and define
\(\mathcal N_1(\xi)=\sum_\alpha c_\alpha\mathcal N_1(Q_\alpha)\).
This gives an \(F\)-linear map
\(\mathcal N_1\colon\bar V\to\mathcal P_1\). Neither the preceding rules
nor the passage to PBW normal form increases PBW length: the first case
replaces one negative mode by one negative mode; the second replaces two
\(L(-2)\)-factors by one \(L(-4)\)-factor; and commutator terms arising
from PBW reordering decrease the number of negative-mode factors.

Whenever an expression of the form \((L(-1)+L(0))\xi\) or
\(u\circ^s_{1,t}v\) is expanded below, the relevant vectors are first
regarded as vectors in \(\bar V\). The mode iterate formula, the Virasoro
commutation relations, and the vacuum condition are then used to obtain a
PBW expansion. The relation \(q=0\) and the congruences in \(O_1(V_q)\)
enter subsequently through the action of \(\mathcal N_1\) and the
argument in the quotient space, respectively.

Before applying \(\mathcal N_1\) to the triangular relations, we verify
that its recursive definition terminates and preserves the represented
class in \(A_1(V_q)\).

\begin{lemma}
\label{lem:A1-normalization-map}
The preceding normalization procedure terminates after finitely many
steps. For every \(Q\in\bar V\), one has
\(\pi_q(Q)\equiv_{A_1}\pi_q(\mathcal N_1(Q))\).
\end{lemma}

\begin{proof}
Let \(Q=L(-i_1)\cdots L(-i_r)\mathbf1\) be a PBW basis vector. After
every word on the right-hand side of the first rule is put into PBW
normal form, each resulting term has total degree strictly smaller than
\(|Q|\). The second rule preserves the total degree, but every term in
its PBW normal form has length strictly smaller than
\(\operatorname{len}(Q)\). Thus each repeated application either
decreases the total degree, or preserves the total degree and decreases
the PBW length. Since both are nonnegative integers, the process
terminates after finitely many steps. A terminal term involves only
\(L(-2),L(-3),L(-4)\) and contains at most one \(L(-2)\)-factor; it
therefore belongs to \(\mathcal P_1\). The uniqueness of the PBW
expansion also shows that the linear extension defining \(\mathcal N_1\)
is a well-defined linear map on \(\bar V\).

Passing to PBW normal form uses identities in \(\bar V\), and hence does
not change the image in \(V_q\). The first rule is the \(n=1\) case of
Lemma~\ref{lem:An-omega-high-index} with \(U=V_q\), so it
preserves the congruence class after projection to \(A_1(V_q)\). The
second rule uses the vacuum relation in \(V_q\): if
\(Q=R L(-2)^2\mathbf1\), then \(\pi_q(Rq)=0\), and consequently
\(\pi_q(Q)=2\pi_q(R L(-4)\mathbf1)\).
Induction first on total degree and then, for a fixed degree, on PBW
length now gives
\(\pi_q(Q)\equiv_{A_1}\pi_q(\mathcal N_1(Q))\).
\end{proof}

We first establish the three triangular families arising from the
translation relation. They handle every nonretained monomial \(P_\lambda\)
except \(P_{(\lambda_4,0,1)}\) with \(\lambda_4\ge2\).

\begin{lemma}
\label{lem:A1-translation-triangular}
Suppose that \(\lambda\in\Lambda_1\) has one of the following three forms:
\(\lambda=(\lambda_4,\lambda_3,0)\) with \(\lambda_3>0\),
\(\lambda=(\lambda_4,0,0)\) with \(\lambda_4>0\), or
\(\lambda=(\lambda_4,\lambda_3,1)\) with \(\lambda_3>0\).
Then there exists
\(\ell_\lambda\in\mathcal P_{<P_\lambda}\) such that
\(P_\lambda\equiv_{A_1}\ell_\lambda\).
\end{lemma}

\begin{proof}
We treat the three forms in order. Empty sums in the following
calculations are understood to be zero.

First consider
\(\lambda=(\lambda_4,\lambda_3,0)\) with \(\lambda_3>0\). To obtain the
target term \(P_{(\lambda_4,\lambda_3,0)}\), take the monomial
\(P_{(\lambda_4,\lambda_3-1,1)}\). When \(L(-1)\) acts on its
\(L(-2)\)-factor, the commutator
\([L(-1),L(-2)]=L(-3)\) produces precisely the target term. For
\(\lambda_3\ge1\), set
\(X=L(-4)^{\lambda_4}L(-3)^{\lambda_3-1}L(-2)\) and expand
\((L(-1)+L(0))P_{(\lambda_4,\lambda_3-1,1)}\):
\[
\begin{aligned}
(L(-1)+L(0))P_{(\lambda_4,\lambda_3-1,1)}
={}&(L(-1)+L(0))X\mathbf1\\
={}&L(-1)X\mathbf1+L(0)X\mathbf1\\
={}&[L(-1),X]\mathbf1
+(4\lambda_4+3\lambda_3-1)P_{(\lambda_4,\lambda_3-1,1)}\\
={}&L(-4)^{\lambda_4}L(-3)^{\lambda_3-1}[L(-1),L(-2)]\mathbf1\\
&+3\sum_{j=0}^{\lambda_4-1}
L(-4)^jL(-5)L(-4)^{\lambda_4-1-j}L(-3)^{\lambda_3-1}L(-2)\mathbf1\\
&+2\sum_{j=0}^{\lambda_3-2}
L(-4)^{\lambda_4}L(-3)^jL(-4)L(-3)^{\lambda_3-2-j}L(-2)\mathbf1\\
&+(4\lambda_4+3\lambda_3-1)P_{(\lambda_4,\lambda_3-1,1)}.
\end{aligned}
\]
The second equality expands the action of \(L(-1)+L(0)\). The third uses
\(L(-1)X\mathbf1=[L(-1),X]\mathbf1\) because
\(L(-1)\mathbf1=0\). The \(L(0)\)-term on the same line is the
weight \(4\lambda_4+3\lambda_3-1\) of the initial vector times that
vector. We then apply the Leibniz rule
\([A,BC]=[A,B]C+B[A,C]\) and
\([L(-1),L(-m)]=(m-1)L(-m-1)\) to the individual factors of \(X\). In
particular,
\(L(-4)^{\lambda_4}L(-3)^{\lambda_3-1}
[L(-1),L(-2)]\mathbf1=P_{(\lambda_4,\lambda_3,0)}\).
Every \(L(-5)\)-factor in the first sum is handled by the \(n=1\) case
of Lemma~\ref{lem:An-omega-high-index} with \(U=V_q\): for every
\(w\in V_q\), in \(A_1(V_q)\),
\[
L(-5)w\equiv_{A_1}
-3L(-4)w-3L(-3)w-L(-2)w.
\]
After applying \(\mathcal N_1\) to the second and third sums, their
length-preserving parts can produce only terms of the forms
\(P_{(\lambda_4,\lambda_3-1,1)}\),
\(P_{(\lambda_4-1,\lambda_3,1)}\), and
\(P_{(\lambda_4+1,\lambda_3-2,1)}\).
Terms containing two \(L(-2)\)-factors decrease in length when the pair
is eliminated, and commutator terms arising from PBW reordering are also
shorter. Here and below, a term \(P_\mu\) is understood to be zero if
one of the coordinates of \(\mu\) is negative. All three terms, as well as the last
term \(P_{(\lambda_4,\lambda_3-1,1)}\), are lower than
\(P_{(\lambda_4,\lambda_3,0)}\). Hence there exists
\(\ell_{(\lambda_4,\lambda_3,0)}
\in\mathcal P_{<P_{(\lambda_4,\lambda_3,0)}}\) such that
\begin{equation}\label{eq:A1-tri-lambda2-zero-lambda3-positive}
\mathcal N_1\bigl((L(-1)+L(0))P_{(\lambda_4,\lambda_3-1,1)}\bigr)
=P_{(\lambda_4,\lambda_3,0)}
-\ell_{(\lambda_4,\lambda_3,0)}
\qquad(\lambda_3\ge1,\ \lambda_4\ge0).
\end{equation}
Since
\(\pi_q((L(-1)+L(0))P_{(\lambda_4,\lambda_3-1,1)})
\in O_1(V_q)\), Lemma~\ref{lem:A1-normalization-map} and
\eqref{eq:A1-tri-lambda2-zero-lambda3-positive} give
\(P_{(\lambda_4,\lambda_3,0)}
\equiv_{A_1}\ell_{(\lambda_4,\lambda_3,0)}\)
in \(A_1(V_q)\). This proves the first case.

Next consider \(\lambda=(\lambda_4,0,0)\) with \(\lambda_4>0\). To obtain
the target term \(P_{(\lambda_4,0,0)}\), take
\(P_{(\lambda_4-1,1,0)}\). Its \(L(-3)\)-factor produces the target term
with coefficient \(2\) under
\([L(-1),L(-3)]=2L(-4)\). For \(\lambda_4\ge1\),
\[
\begin{aligned}
(L(-1)+L(0))P_{(\lambda_4-1,1,0)}
={}&2P_{(\lambda_4,0,0)}
+3\sum_{j=0}^{\lambda_4-2}
L(-4)^jL(-5)L(-4)^{\lambda_4-2-j}L(-3)\mathbf1\\
&+(4\lambda_4-1)P_{(\lambda_4-1,1,0)}.
\end{aligned}
\]
After \(\mathcal N_1\) is applied to the \(L(-5)\)-terms in the sum, the
length-preserving part can produce only
\(P_{(\lambda_4-1,1,0)}\), \(P_{(\lambda_4-2,2,0)}\), and
\(P_{(\lambda_4-2,1,1)}\),
whereas commutator terms arising from PBW reordering have smaller length.
These terms, together with the final
\(P_{(\lambda_4-1,1,0)}\), are all strictly lower than
\(P_{(\lambda_4,0,0)}\). Thus there exists
\(\ell_{(\lambda_4,0,0)}
\in\mathcal P_{<P_{(\lambda_4,0,0)}}\) such that
\begin{equation}\label{eq:A1-tri-lambda2-zero-lambda3-zero}
\mathcal N_1\bigl((L(-1)+L(0))P_{(\lambda_4-1,1,0)}\bigr)
=2P_{(\lambda_4,0,0)}-2\ell_{(\lambda_4,0,0)}
\qquad(\lambda_4\ge1).
\end{equation}
Since
\(\pi_q((L(-1)+L(0))P_{(\lambda_4-1,1,0)})\in O_1(V_q)\),
Lemma~\ref{lem:A1-normalization-map} and
\eqref{eq:A1-tri-lambda2-zero-lambda3-zero} give
\(P_{(\lambda_4,0,0)}
\equiv_{A_1}\ell_{(\lambda_4,0,0)}\)
in \(A_1(V_q)\). This proves the second case.

Finally, consider
\(\lambda=(\lambda_4,\lambda_3,1)\) with \(\lambda_3>0\). To obtain the
target term \(P_{(\lambda_4,\lambda_3,1)}\), use the auxiliary vector on
the left-hand side below, which contains two \(L(-2)\)-factors. The action
of \(L(-1)\) on these two factors gives the leading term
\(2P_{(\lambda_4,\lambda_3,1)}\). Expanding the action of
\(L(-1)+L(0)\) gives
\[
\begin{aligned}
&(L(-1)+L(0))L(-4)^{\lambda_4}
L(-3)^{\lambda_3-1}L(-2)^2\mathbf1\\
={}&L(-4)^{\lambda_4}L(-3)^{\lambda_3-1}
\bigl([L(-1),L(-2)]L(-2)+L(-2)[L(-1),L(-2)]\bigr)\mathbf1\\
&+3\sum_{j=0}^{\lambda_4-1}
L(-4)^jL(-5)L(-4)^{\lambda_4-1-j}
L(-3)^{\lambda_3-1}L(-2)^2\mathbf1\\
&+2\sum_{j=0}^{\lambda_3-2}
L(-4)^{\lambda_4}L(-3)^jL(-4)
L(-3)^{\lambda_3-2-j}L(-2)^2\mathbf1\\
&+(4\lambda_4+3\lambda_3+1)
L(-4)^{\lambda_4}L(-3)^{\lambda_3-1}L(-2)^2\mathbf1.
\end{aligned}
\]
Set \(d=\lambda_3+\lambda_4\). Using
\(L(-2)L(-3)=L(-3)L(-2)+L(-5)\),
the word containing the two commutator terms on the right-hand side is
the sum of \(2P_{(\lambda_4,\lambda_3,1)}\) and
\(L(-4)^{\lambda_4}L(-3)^{\lambda_3-1}L(-5)\mathbf1\).
The latter word has length \(d\), which does not increase under PBW
normalization or the action of \(\mathcal N_1\). Both sums and the final
\(L(0)\)-weight term contain two \(L(-2)\)-factors. Under
\(\mathcal N_1\), either PBW reordering first decreases the length, or
the pair of \(L(-2)\)-factors is eliminated and the length decreases from
\(d+1\) to at most \(d\). By the second assertion of
Lemma~\ref{lem:A1-length-order}, all nonleading terms above belong to
\(\mathcal P_{<P_{(\lambda_4,\lambda_3,1)}}\). There is therefore an
\(\ell_{(\lambda_4,\lambda_3,1)}
\in\mathcal P_{<P_{(\lambda_4,\lambda_3,1)}}\) such that
\begin{equation}\label{eq:A1-tri-lambda2-one-lambda3-positive}
\begin{aligned}
&\mathcal N_1\Bigl((L(-1)+L(0))L(-4)^{\lambda_4}
L(-3)^{\lambda_3-1}L(-2)^2\mathbf1\Bigr)\\
&\qquad=2P_{(\lambda_4,\lambda_3,1)}
-2\ell_{(\lambda_4,\lambda_3,1)}
\qquad(\lambda_3\ge1,\ \lambda_4\ge0).
\end{aligned}
\end{equation}
Since
\[
\pi_q\bigl((L(-1)+L(0))L(-4)^{\lambda_4}
L(-3)^{\lambda_3-1}L(-2)^2\mathbf1\bigr)\in O_1(V_q),
\]
Lemma~\ref{lem:A1-normalization-map} and
\eqref{eq:A1-tri-lambda2-one-lambda3-positive} give
\(P_{(\lambda_4,\lambda_3,1)}
\equiv_{A_1}\ell_{(\lambda_4,\lambda_3,1)}\)
in \(A_1(V_q)\). This proves the third case.
\end{proof}

We next determine the leading term of the circle-product relation for
the remaining monomials \(P_{(\lambda_4,0,1)}\).

\begin{lemma}
\label{lem:A1-pure-L4-leading}
For \(\lambda_4\ge2\), let \(u=L(-4)L(-2)\mathbf1\) and
\(t_{\lambda_4}=L(-4)^{\lambda_4-2}L(-2)\mathbf1\).
Then there exists
\(\ell_{(\lambda_4,0,1)}
\in\mathcal P_{<P_{(\lambda_4,0,1)}}\) such that
\begin{equation}\label{eq:A1-tri-lambda2-one-lambda3-zero}
\mathcal N_1(u\circ^0_{1,0}t_{\lambda_4})
=2P_{(\lambda_4,0,1)}-2\ell_{(\lambda_4,0,1)}.
\end{equation}
\end{lemma}

\begin{proof}
Fix \(\lambda_4\ge2\) and calculate
\(u\circ^0_{1,0}t_{\lambda_4}\) in \(\bar V\). Since
\(\operatorname{wt}(u)=6\), the identity
\((1+z)^7=1+z^7\) holds in characteristic \(7\). Hence
\[
u\circ^0_{1,0}t_{\lambda_4}
=
\operatorname{Res}_zY(u,z)t_{\lambda_4}\,
\frac{(1+z)^7}{z^4}
=u_{-4}t_{\lambda_4}+u_3t_{\lambda_4}.
\]
Here \(u=\omega_{-3}\omega\). The mode iterate formula obtained from the
Jacobi identity in the arbitrary-field setting~\citep{DR14} is
\[
(\omega_{-3}\omega)_k
=
\sum_{i\ge0}(-1)^i\binom{-3}{i}
\bigl(\omega_{-3-i}\omega_{k+i}
{}+\omega_{k-3-i}\omega_i\bigr).
\]
Since \(\omega_j=L(j-1)\), taking \(k=-4\) and \(k=3\), respectively,
gives
\[
\begin{aligned}
u_{-4}t_{\lambda_4}
={}&\sum_{i\ge0}(-1)^i\binom{-3}{i}
\Bigl(L(-4-i)L(i-5)t_{\lambda_4}
+L(-8-i)L(i-1)t_{\lambda_4}\Bigr),\\
u_3t_{\lambda_4}
={}&\sum_{i\ge0}(-1)^i\binom{-3}{i}
\Bigl(L(-4-i)L(i+2)t_{\lambda_4}
+L(-1-i)L(i-1)t_{\lambda_4}\Bigr).
\end{aligned}
\]
By the lower truncation property, only finitely many summands remain
nonzero after acting on the fixed vector \(t_{\lambda_4}\). For the
present target \(P_{(\lambda_4,0,1)}\), the space
\(\mathcal P_{<P_{(\lambda_4,0,1)}}\) is precisely the span of all
\(P_{(\mu_4,\mu_3,0)}\) with
\(\mu_3+\mu_4\le\lambda_4\) and all \(P_{(\mu_4,\mu_3,1)}\) with
\(\mu_3,\mu_4\ge0\),
\(\mu_3+\mu_4\le\lambda_4\), and
\((\mu_3,\mu_4)\ne(0,\lambda_4)\).

Denote the two sums in \(u_{-4}t_{\lambda_4}\) by
\[
S_{\lambda_4}^{(1)}=\sum_{i\ge0}(-1)^i\binom{-3}{i}L(-4-i)L(i-5)t_{\lambda_4},
\qquad
S_{\lambda_4}^{(2)}=\sum_{i\ge0}(-1)^i\binom{-3}{i}L(-8-i)L(i-1)t_{\lambda_4}.
\]
By definition of \(\nu\), the ordering parameter of
\(P_{(\lambda_4,0,1)}\) is \((\lambda_4,1,0)\). For \(P_\mu\), each of
the following conditions implies
\(P_\mu\in\mathcal P_{<P_{(\lambda_4,0,1)}}\): its first ordering
coordinate is smaller than \(\lambda_4\); its first ordering coordinate
equals \(\lambda_4\) and \(\mu_2=0\); or its first ordering coordinate
equals \(\lambda_4\), \(\mu_2=1\), and \(\mu_3>0\). We calculate
\(S_{\lambda_4}^{(1)}\), \(S_{\lambda_4}^{(2)}\), and
\(u_3t_{\lambda_4}\) in turn. The calculation gives coefficient \(2\)
for the target term and places every remaining term in
\(\mathcal P_{<P_{(\lambda_4,0,1)}}\).

To expand \(S_{\lambda_4}^{(1)}\), set \(r=\lambda_4-2\) and
\(c_i=(-1)^i\binom{-3}{i}\).
Move \(L(i-5)\) successively past the \(r\) copies of \(L(-4)\) and the
final \(L(-2)\), using
\[
[L(i-5),L(-4)]=(i-1)L(i-9)+6\delta_{i,9},\qquad
[L(i-5),L(-2)]=(i-3)L(i-7)+2\delta_{i,7}.
\]
Here \(\delta_{i,j}\) denotes the Kronecker delta: it equals \(1\) when
\(i=j\) and \(0\) otherwise. The coefficients of the central terms have
already been evaluated at the central charge \(c=4\) and simplified in
\(F\). We obtain
\[
\begin{aligned}
S_{\lambda_4}^{(1)}
={}&\sum_{i\ge0}c_iL(-4-i)L(i-5)
L(-4)^rL(-2)\mathbf1\\
={}&\sum_{i\ge0}\sum_{j=0}^{r-1}c_iL(-4-i)L(-4)^j
 \bigl((i-1)L(i-9)+6\delta_{i,9}\bigr)
 L(-4)^{r-1-j}L(-2)\mathbf1\\
&\quad{}+\sum_{i\ge0}c_iL(-4-i)L(-4)^r
 \bigl((i-3)L(i-7)+2\delta_{i,7}\bigr)\mathbf1\\
&\quad{}+\sum_{i=0}^{3}c_iL(-4-i)L(-4)^rL(-2)L(i-5)\mathbf1.
\end{aligned}
\]
For \(r=0\), the inner sum over \(j\) is understood to be empty. Only
\(0\le i\le3\) is required in the last sum, since
\(L(i-5)\mathbf1=0\) for \(i\ge4\). Recall that neither
\(\mathcal N_1\) nor PBW normalization increases PBW length:
Lemma~\ref{lem:An-omega-high-index} replaces one negative mode by one
negative mode, a commutator term in PBW reordering decreases the length,
and the elimination of a pair of \(L(-2)\)-factors replaces two negative
modes by one.

The first two sums in the second equality are the commutator
contributions. The commutator of \(L(i-5)\) with one of the \(r\)
copies of \(L(-4)\)
produces words of the form
\[
L(-4-i)L(-4)^jL(i-9)
L(-4)^{r-1-j}L(-2)\mathbf1,
\qquad 0\le j\le r-1.
\]
When \(L(i-9)\) is a negative mode, such a word has length
\(r+2=\lambda_4\). When it is a nonnegative mode, move it to the right
and let it act on the vacuum vector. The first assertion of
Lemma~\ref{lem:A1-length-order} shows that the resulting PBW monomials
have no more negative-mode factors, while the central term arising from
\(\delta_{i,9}\) has smaller length.

The commutator of \(L(i-5)\) with the final \(L(-2)\) produces words of
the form \(L(-4-i)L(-4)^rL(i-7)\mathbf1\).
If \(L(i-7)\) is a negative mode, the length is
\(r+2=\lambda_4\); if it is nonnegative, move it to the right and let it
act on the vacuum vector. By the first assertion of
Lemma~\ref{lem:A1-length-order}, every resulting term has length at most
\(\lambda_4\), while the central term arising from \(\delta_{i,7}\) has
smaller length. It follows that every \(P_\mu\) obtained from these
words after applying \(\mathcal N_1\) satisfies
\(\operatorname{len}(P_\mu)\le\lambda_4\).
By Lemma~\ref{lem:A1-length-order}, all such \(P_\mu\) belong to
\(\mathcal P_{<P_{(\lambda_4,0,1)}}\).

It remains to consider
\(c_iL(-4-i)L(-4)^rL(-2)L(i-5)\mathbf1\),
because only these words can have length \(\lambda_4+1\). For \(i\ge2\),
only \(i=2,3\) occur. When \(i=3\), the corresponding word is
\(c_3L(-7)L(-4)^rL(-2)^2\mathbf1\).
After eliminating the pair of \(L(-2)\)-factors, its length is at most
\(\lambda_4\). Further applications of \(\mathcal N_1\) do not increase
the length, so Lemma~\ref{lem:A1-length-order} places all resulting terms
in \(\mathcal P_{<P_{(\lambda_4,0,1)}}\).

For \(i=2\), first interchange the final \(L(-2)\) and \(L(-3)\):
\[
c_2L(-6)L(-4)^rL(-2)L(-3)\mathbf1
=c_2L(-6)L(-4)^rL(-3)L(-2)\mathbf1
+c_2L(-6)L(-4)^rL(-5)\mathbf1.
\]
The second term on the right has length \(r+2=\lambda_4\). Since PBW
normalization and \(\mathcal N_1\) do not increase the length,
Lemma~\ref{lem:A1-length-order} shows that all terms obtained from it
belong to \(\mathcal P_{<P_{(\lambda_4,0,1)}}\). Applying the first rule
of \(\mathcal N_1\) to the first term gives
\[
\begin{aligned}
\mathcal N_1\bigl(c_2L(-6)L(-4)^rL(-3)L(-2)\mathbf1\bigr)
={}&-3c_2\mathcal N_1\bigl(L(-5)L(-4)^rL(-3)L(-2)\mathbf1\bigr)\\
&-3c_2\mathcal N_1\bigl(
L(-4)^{r+1}L(-3)L(-2)\mathbf1\bigr)\\
&-c_2\mathcal N_1\bigl(
L(-3)L(-4)^rL(-3)L(-2)\mathbf1\bigr).
\end{aligned}
\]
Every term on the right arising from a commutator or from the elimination
of a pair of \(L(-2)\)-factors has length at most \(\lambda_4\), and hence
belongs to \(\mathcal P_{<P_{(\lambda_4,0,1)}}\) by
Lemma~\ref{lem:A1-length-order}. Every remaining term of length
\(\lambda_4+1\) contains exactly one \(L(-2)\)-factor and at least one
\(L(-3)\)-factor. Along any branch on which the length remains
\(\lambda_4+1\), neither a commutator term from PBW reordering nor the
elimination of a pair of \(L(-2)\)-factors can occur, since either
operation lowers the length. The high-index recurrence replaces one
mode by one mode, and its length-preserving PBW branch retains the
pre-existing \(L(-3)\)-factor. Hence each monomial \(P_\mu\) obtained
from these terms under \(\mathcal N_1\) satisfies
\(\mu_3+\mu_4=\lambda_4\), \(\mu_2=1\), and \(\mu_3>0\).

In summary, every monomial \(P_\mu\) obtained from the terms with
\(i\ge2\) either satisfies
\(\operatorname{len}(P_\mu)\le\lambda_4\), or satisfies
\(\operatorname{len}(P_\mu)=\lambda_4+1\), \(\mu_2=1\), and
\(\mu_3>0\).
The former belongs to \(\mathcal P_{<P_{(\lambda_4,0,1)}}\) by
Lemma~\ref{lem:A1-length-order}; in the latter case,
\(\nu(P_\mu)=(\lambda_4,1,-\mu_3)<(\lambda_4,1,0)\),
which gives the same conclusion. Therefore only the terms with
\(i=0,1\) can contribute the target term
\(P_{(\lambda_4,0,1)}\). In \(S_{\lambda_4}^{(1)}\), these two terms are
\(c_0L(-4)L(-5)t_{\lambda_4}\) and
\(c_1L(-5)L(-4)t_{\lambda_4}\).

We now calculate the total coefficient of the target term contributed by
these two summands. Since
\(c_i=(-1)^i\binom{-3}{i}\), in \(F\) we have
\(c_0=(-1)^0\binom{-3}{0}=1\) and
\(c_1=(-1)^1\binom{-3}{1}=3\).
Moreover,
\[
c_0L(-4)L(-5)t_{\lambda_4}+c_1L(-5)L(-4)t_{\lambda_4}
=4L(-5)L(-4)t_{\lambda_4}+L(-9)t_{\lambda_4},
\]
where we used
\(L(-4)L(-5)=L(-5)L(-4)+L(-9)\). The word
\(L(-9)t_{\lambda_4}\) has length \(\lambda_4\), and
\(\mathcal N_1\) does not increase the length. By
Lemma~\ref{lem:A1-length-order}, it contributes only lower terms after
\(\mathcal N_1\) is applied. In the other term, apply the rule for
\(L(-5)\). Only
\(-3L(-4)L(-4)t_{\lambda_4}\) contributes the target term
\(P_{(\lambda_4,0,1)}\). Among the terms containing \(L(-3)\), the only
length-preserving term after PBW normalization is
\(P_{(\lambda_4-1,1,1)}\), whose third ordering coordinate is
\(-1<0\). Commutator terms from the reordering have smaller length and
are lower by Lemma~\ref{lem:A1-length-order}. Terms containing two
\(L(-2)\)-factors are likewise lower after the pair is eliminated.
Consequently,
\[
\mathcal N_1(S_{\lambda_4}^{(1)})
\equiv 4(-3)P_{(\lambda_4,0,1)}
=2P_{(\lambda_4,0,1)}
\pmod{\mathcal P_{<P_{(\lambda_4,0,1)}}}.
\]

We next consider \(S_{\lambda_4}^{(2)}\). Move \(L(i-1)\) successively
past the negative modes in
\(t_{\lambda_4}=L(-4)^rL(-2)\mathbf1\). The relations are
\[
[L(i-1),L(-4)]=(i+3)L(i-5)+6\delta_{i,5},\qquad
[L(i-1),L(-2)]=(i+1)L(i-3)+2\delta_{i,3},
\]
where the central terms have again been evaluated at \(c=4\) and
simplified in \(F\). Together with
\(L(i-1)\mathbf1=0\) for \(i\ge0\), these relations give
\[
\begin{aligned}
L(i-1)t_{\lambda_4}
={}&\sum_{j=0}^{r-1}L(-4)^j
\bigl((i+3)L(i-5)+6\delta_{i,5}\bigr)
L(-4)^{r-1-j}L(-2)\mathbf1\\
&+L(-4)^r\bigl((i+1)L(i-3)+2\delta_{i,3}\bigr)\mathbf1.
\end{aligned}
\]
After substitution into \(S_{\lambda_4}^{(2)}\), the noncentral terms
have only the following two forms:
\[
\begin{aligned}
&L(-8-i)L(-4)^jL(i-5)L(-4)^{r-1-j}L(-2)\mathbf1,
&&0\le j\le r-1,\\
&L(-8-i)L(-4)^rL(i-3)\mathbf1.
\end{aligned}
\]
When \(L(i-5)\) or \(L(i-3)\) is a negative mode, these words have length
\(r+2=\lambda_4\). When it is a nonnegative mode, moving it to the right
until it acts on the vacuum vector produces terms with no more
negative-mode factors, by the first assertion of
Lemma~\ref{lem:A1-length-order}; central terms have strictly smaller
length. Since a subsequent application of \(\mathcal N_1\) does not
increase PBW length, every monomial \(P_\mu\) occurring in
\(\mathcal N_1(S_{\lambda_4}^{(2)})\) has length at most
\(\lambda_4\). The second assertion of
Lemma~\ref{lem:A1-length-order} therefore gives
\[
\mathcal N_1(S_{\lambda_4}^{(2)})
\equiv0\pmod{\mathcal P_{<P_{(\lambda_4,0,1)}}}.
\]

Finally, consider \(u_3t_{\lambda_4}\). For a word
\(L(-4-i)L(i+2)t_{\lambda_4}\),
the positive mode \(L(i+2)\) acts to the right on \(t_{\lambda_4}\).
Since its action on the vacuum vector vanishes, only its commutators with
the negative modes in \(t_{\lambda_4}\) remain. Each noncentral
commutator replaces one negative mode by one Virasoro operator, while a
central term decreases the length further. Equivalently, the first
assertion of Lemma~\ref{lem:A1-length-order}, followed by the inclusion
of the outer \(L(-4-i)\), bounds the resulting lengths by
\(\operatorname{len}(t_{\lambda_4})+1=\lambda_4\).

Now consider \(L(-1-i)L(i-1)t_{\lambda_4}\).
For \(i=0\), the two \(L(-1)\)-operators act successively to the right and
vanish on the vacuum vector; the resulting words have length at most
\(\operatorname{len}(t_{\lambda_4})=\lambda_4-1\). For \(i=1\),
\(t_{\lambda_4}\) is homogeneous, so \(L(0)t_{\lambda_4}\) is a scalar
multiple of \(t_{\lambda_4}\); the outer \(L(-2)\) therefore increases
the length to at most \(\lambda_4\). For \(i\ge2\),
\(L(i-1)\) is a positive mode. As in the preceding case, moving it to the
right and then including the outer \(L(-1-i)\) again yields words of
length at most \(\lambda_4\), by the first assertion of
Lemma~\ref{lem:A1-length-order}. Hence every monomial \(P_\mu\) obtained
from \(u_3t_{\lambda_4}\) after PBW normalization and application of
\(\mathcal N_1\) has length at most \(\lambda_4\). By the second
assertion of Lemma~\ref{lem:A1-length-order},
\[
\mathcal N_1(u_3t_{\lambda_4})\equiv0
\pmod{\mathcal P_{<P_{(\lambda_4,0,1)}}}.
\]

Since
\(u_{-4}t_{\lambda_4}
=S_{\lambda_4}^{(1)}+S_{\lambda_4}^{(2)}\), the three congruences above
combine to give
\[
\mathcal N_1(u_{-4}t_{\lambda_4}+u_3t_{\lambda_4})
\equiv2P_{(\lambda_4,0,1)}
\pmod{\mathcal P_{<P_{(\lambda_4,0,1)}}}.
\]
Together with
\(u\circ^0_{1,0}t_{\lambda_4}
=u_{-4}t_{\lambda_4}+u_3t_{\lambda_4}\), and using the invertibility of
\(2\) in \(F\), this allows us to choose
\(\ell_{(\lambda_4,0,1)}
\in\mathcal P_{<P_{(\lambda_4,0,1)}}\) so that
\eqref{eq:A1-tri-lambda2-one-lambda3-zero} holds.
\end{proof}

The three translation families and the calculation for
\(P_{(\lambda_4,0,1)}\) now combine into a triangular relation for every
monomial \(P_\lambda\) outside the retained set.

\begin{theorem}
\label{thm:A1-PBW-triangular}
If \(\lambda\in\Lambda_1\setminus
\{(0,0,0),(0,0,1),(1,0,1)\}\),
then there exists
\(\ell_\lambda\in\mathcal P_{<P_\lambda}\) such that
\(P_\lambda\equiv_{A_1}\ell_\lambda\).
\end{theorem}

\begin{proof}
Every \(\lambda\) satisfying the hypothesis has exactly one of four
forms. The first three forms are \((\lambda_4,\lambda_3,0)\) with
\(\lambda_3>0\), \((\lambda_4,0,0)\) with \(\lambda_4>0\), and
\((\lambda_4,\lambda_3,1)\) with \(\lambda_3>0\); they
are covered by Lemma~\ref{lem:A1-translation-triangular}.

The remaining case is
\(\lambda=(\lambda_4,0,1)\) with \(\lambda_4\ge2\). Take
\(u=L(-4)L(-2)\mathbf1\) and
\(t_{\lambda_4}=L(-4)^{\lambda_4-2}L(-2)\mathbf1\).
By Lemma~\ref{lem:A1-pure-L4-leading}, there exists
\(\ell_{(\lambda_4,0,1)}
\in\mathcal P_{<P_{(\lambda_4,0,1)}}\) such that
\eqref{eq:A1-tri-lambda2-one-lambda3-zero} holds. On the other hand, the
homomorphism property of \(\pi_q\) and the definition of \(O_1(V_q)\)
give \(\pi_q(u\circ^0_{1,0}t_{\lambda_4})\in O_1(V_q)\).
Applying \(\pi_q\) to
\eqref{eq:A1-tri-lambda2-one-lambda3-zero} and
using Lemma~\ref{lem:A1-normalization-map}, we obtain
\(0\equiv_{A_1}2P_{(\lambda_4,0,1)}-2\ell_{(\lambda_4,0,1)}\).
Since \(2\) is invertible in \(F\), it follows that
\(P_{(\lambda_4,0,1)}
\equiv_{A_1}\ell_{(\lambda_4,0,1)}\).
Thus the fourth form also satisfies the assertion.
\end{proof}

Iterating the preceding triangular relations along the order determined
by \(\nu\) gives the global three-dimensional spanning bound.

\begin{lemma}
\label{lem:A1-global-pbw-reduction}
The class of every vacuum PBW monomial in \(A_1(V_q)\) belongs to
\[
\operatorname{span}_F
\bigl\{[\mathbf1]_1,[L(-2)\mathbf1]_1,
[L(-4)L(-2)\mathbf1]_1\bigr\}.
\]
\end{lemma}

\begin{proof}
The map \(\pi_q\colon\bar V\to V_q\) is the natural quotient map and is
therefore surjective. Moreover,
the vacuum PBW monomials form a basis of \(\bar V\). It therefore
suffices to consider their images. Let \(Q\in\bar V\) be a vacuum PBW
monomial. By Lemma~\ref{lem:A1-normalization-map},
\(\pi_q(Q)\equiv_{A_1}\pi_q(\mathcal N_1(Q))\), and
\(\mathcal N_1(Q)\in\mathcal P_1=
\operatorname{span}_F\{P_\lambda:\lambda\in\Lambda_1\}\).
Thus the class of every vacuum PBW monomial is a linear combination of
the monomials \(P_\lambda\).

We now use strong induction on \(P_\lambda\) with respect to the
lexicographic order determined by \(\nu\). Its first coordinate
\(\lambda_3+\lambda_4\) is a nonnegative integer, and only finitely many
\(P_\lambda\) have a fixed first coordinate. Hence the order is well
founded. By definition, the three retained monomials
\(P_{(0,0,0)}\), \(P_{(0,0,1)}\), and \(P_{(1,0,1)}\)
belong to \(\mathcal B_1\). Suppose that every monomial \(P_\mu\)
strictly lower than a fixed nonretained monomial \(P_\lambda\) is
congruent in \(A_1(V_q)\)
to an element of \(\mathcal B_1\). By
Theorem~\ref{thm:A1-PBW-triangular}, there exists
\(\ell_\lambda\in\mathcal P_{<P_\lambda}\) such that
\(P_\lambda\equiv_{A_1}\ell_\lambda\).
Every term of \(\ell_\lambda\) is strictly lower than \(P_\lambda\).
The induction hypothesis therefore gives an element
\(b_\lambda\in\mathcal B_1\) such that
\(P_\lambda\equiv_{A_1}b_\lambda\).
Consequently, every monomial \(P_\lambda\) is congruent in \(A_1(V_q)\)
to an element of \(\mathcal B_1\), and the class of every vacuum PBW
monomial belongs to the stated three-dimensional space.
\end{proof}

The spanning result gives the required vector-space upper bound. To
express the retained PBW representatives in terms of
\(x_q=[\omega]_1\), we now compute the product \(x_q*_1x_q\).

\begin{proposition}
\label{prop:A1-upper}
Let \(x_q=[\omega]_1\in A_1(V_q)\). Then
\(A_1(V_q)=\operatorname{span}_F\{1,x_q,x_q^2\}\), where
\(x_q^2=x_q*_1x_q\).
The natural quotient map \(V_q\to V\) induces a surjective algebra
homomorphism \(A_1(V_q)\to A_1(V)\), and hence
\(A_1(V)=\operatorname{span}_F\{1,x,x^2\}\). In particular,
\(\dim_F A_1(V_q)\le3\) and \(\dim_F A_1(V)\le3\).
\end{proposition}

\begin{proof}
By Lemma~\ref{lem:A1-global-pbw-reduction}, the vector space \(A_1(V_q)\)
is spanned by \([\mathbf1]_1\), \([L(-2)\mathbf1]_1\), and
\([L(-4)L(-2)\mathbf1]_1\).
The first two elements are \(1\) and \(x_q\), respectively. We calculate
the PBW coordinates of \(x_q^2=x_q*_1x_q\).

Using the product defining \(A_n(V)\) and
\(\operatorname{wt}(\omega)=2\), for \(n=1\) we have
\[
\omega*_1\omega
=
\sum_{m=0}^{1}(-1)^m\binom{m+1}{1}
\operatorname{Res}_zY(\omega,z)\omega
\frac{(1+z)^3}{z^{m+2}}.
\]
Expanding the residue by
\(Y(\omega,z)=\sum_{n\in\mathbb Z}L(n)z^{-n-2}\) gives the PBW
expression
\[
\omega*_1\omega
=
2L(-2)\mathbf1+L(-3)\mathbf1-3L(-2)^2\mathbf1
+2L(-3)L(-2)\mathbf1-2L(-4)L(-2)\mathbf1.
\]

We simplify the low-weight terms in \(A_1(V_q)\); all congruences below
are modulo \(O_1(V_q)\). Applying the translation relation to \(\omega\)
gives \(L(-3)\mathbf1+2L(-2)\mathbf1\equiv0\).
The relation
\(q=(L(-2)^2-2L(-4))\mathbf1=0\) and the translation relation applied
to \(L(-3)\mathbf1\) give, respectively,
\(L(-2)^2\mathbf1\equiv2L(-4)\mathbf1\) and
\(2L(-4)\mathbf1+3L(-3)\mathbf1\equiv0\). It follows that
\(L(-2)^2\mathbf1\equiv-L(-2)\mathbf1\).

Next apply the translation relation to \(L(-2)^2\mathbf1\) and put the
result into PBW normal form. This gives
\(2L(-3)L(-2)\mathbf1+L(-5)\mathbf1
+4L(-2)^2\mathbf1\equiv0\).
The \(n=1\) case of Lemma~\ref{lem:An-omega-high-index}, applied to
\(u=\mathbf1\), also gives
\(L(-5)\mathbf1+3L(-4)\mathbf1+3L(-3)\mathbf1
+L(-2)\mathbf1\equiv0\).
Thus \(L(-5)\mathbf1\equiv3L(-2)\mathbf1\), and substitution into the
preceding relation yields
\(L(-3)L(-2)\mathbf1\equiv4L(-2)\mathbf1\).
Substituting these identities into the expansion of
\(\omega*_1\omega\), we obtain
\[
x_q*_1x_q
=-3[L(-2)\mathbf1]_1-2[L(-4)L(-2)\mathbf1]_1.
\]
Here \(x_q^2\) denotes the algebra product \(x_q*_1x_q\) displayed above.
Since \(-2\in F_7^\times\), the third PBW representative
\([L(-4)L(-2)\mathbf1]_1\) is an \(F\)-linear combination of \(x_q\) and
\(x_q^2\). Hence
\(A_1(V_q)=\operatorname{span}_F\{1,x_q,x_q^2\}\).

By the functoriality under quotient maps noted above, the natural
quotient map \(\pi_V\colon V_q\to V\) induces a surjective algebra
homomorphism \(A_1(V_q)\to A_1(V)\). The image of \(x_q\) is \(x\), and
therefore \(A_1(V)=\operatorname{span}_F\{1,x,x^2\}\).
\end{proof}

To show that the upper bound is sharp, we use the zero-mode actions on
the three one-dimensional homogeneous subspaces identified above.

\begin{proposition}
\label{prop:A1-lower}
The following algebra homomorphism is surjective:
\[
\rho_1:A_1(V)\to
\operatorname{End}_F(\mathcal L_0(0))
\oplus \operatorname{End}_F(\mathcal L_4(0))
\oplus \operatorname{End}_F(\mathcal L_4(1))
\cong F^3.
\]
In particular, \(\dim_F A_1(V)\ge3\).
\end{proposition}

\begin{proof}
By Theorem~\ref{thm:modular-An-well-defined},
\(\mathcal L_0(0)\), \(\mathcal L_4(0)\), and \(\mathcal L_4(1)\) are
\(A_1(V)\)-modules under the zero-mode action. By
Lemma~\ref{lem:low-layers}, all three spaces are one-dimensional.
Therefore, for \(\xi\in V\),
\[
\rho_1([\xi]_1)=
\bigl(o(\xi)|_{\mathcal L_0(0)},
o(\xi)|_{\mathcal L_4(0)},o(\xi)|_{\mathcal L_4(1)}\bigr)
\]
defines a well-defined algebra homomorphism
\(\rho_1:A_1(V)\to F^3\).

Since \(o(\omega)=L(0)\), it acts on \(\mathcal L_h(m)\) by the scalar
\(h+m\). Thus, for \(x=[\omega]_1\), one has
\(\rho_1(x)=(0,4,5)\).
These three elements are pairwise distinct in the field \(F\) of
characteristic \(7\), so the map
\(f(T)\longmapsto(f(0),f(4),f(5))\)
defines a surjective algebra homomorphism from \(F[T]\) to \(F^3\). The image under
\(\rho_1\) of the subalgebra generated by \(x\) is therefore \(F^3\).
Hence \(\rho_1(A_1(V))=F^3\), and
\(\dim_F A_1(V)\ge3\).
\end{proof}

We now assemble the preceding calculations to obtain the main result.

\begin{theorem}
\label{thm:main}
Let \(F\) be an algebraically closed field of characteristic \(7\), let
\(\bar V=\bar V(\frac12,0)_F\), and set
\(q=(L(-2)^2-2L(-4))\mathbf1\).
Let \(I(q)\) be the graded ideal generated by \(q\), set
\(V_q=\bar V/I(q)\), and let \(V=L(\frac12,0)_F\) be the simple quotient
of \(\bar V\). Then the natural quotient map
\(\pi_V:V_q\to V\) induces an algebra isomorphism
\(A_1(\pi_V):A_1(V_q)\to A_1(V)\), and
\[
A_1(V_q)\cong A_1(V)
\cong F[T]/(T(T-4)(T-5))\cong F^3.
\]
In particular, both \(A_1(V_q)\) and \(A_1(V)\) are
three-dimensional semisimple commutative algebras.
\end{theorem}

\begin{proof}
Propositions~\ref{prop:A1-upper} and~\ref{prop:A1-lower} give
\[
3\le\dim_F A_1(V)\le\dim_F A_1(V_q)\le3.
\]
Thus both algebras are three-dimensional, and the surjective
homomorphism \(A_1(\pi_V):A_1(V_q)\to A_1(V)\)
established in Proposition~\ref{prop:A1-upper} is an isomorphism. The
surjective algebra homomorphism \(\rho_1:A_1(V)\to F^3\) in
Proposition~\ref{prop:A1-lower} is also an isomorphism, since its source
and target have the same dimension.

Define \(\phi:F[T]\to A_1(V)\) by \(T\mapsto x\).
Proposition~\ref{prop:A1-upper} gives
\(A_1(V)=\operatorname{span}_F\{1,x,x^2\}\), so \(\phi\) is surjective.
The composite \(\rho_1\circ\phi:F[T]\to F^3\) is evaluation
at the three points \(0,4,5\), and hence
\(\ker(\rho_1\circ\phi)=(T(T-4)(T-5))\).
Since \(\rho_1\) is an isomorphism,
\(\ker\phi=\ker(\rho_1\circ\phi)\). It follows that
\(A_1(V)\cong F[T]/(T(T-4)(T-5))\).
The polynomial \(T(T-4)(T-5)\) has three pairwise distinct roots in
\(F\). The Chinese remainder theorem therefore gives
\(F[T]/(T(T-4)(T-5))\cong F^3\).
\end{proof}

\section*{Declaration of generative AI use}

During the preparation of this manuscript, the author used OpenAI Codex
for English translation, language editing, and manuscript formatting.
The author reviewed and edited all generated content and takes full
responsibility for the content of the article.

\end{document}